\pdfoutput=1
\documentclass{article}
\usepackage[a4paper, margin=1.5in]{geometry}

\usepackage[T1]{fontenc}      
\usepackage[utf8]{inputenc}   
\usepackage{microtype}        
\usepackage{mlmodern}

\usepackage{amsmath}
\usepackage{amsthm}
\usepackage{amssymb}
\usepackage{mathtools}
\usepackage{thmtools}
\usepackage{commath}
\usepackage{xfrac}
\usepackage[numbers]{natbib}

\numberwithin{equation}{section}
\declaretheorem[Refname={Theorem,Theorems}]{theorem}
\numberwithin{theorem}{section} 

\declaretheorem[style=definition,numberlike=theorem,Refname={Remark,Remarks}]{remark}
\declaretheorem[numberlike=theorem,Refname={Lemma,Lemmas}]{lemma}

\declaretheorem[name=Proposition,numberlike=theorem,Refname={Proposition,Propositions}]{proposition}

\newcommand{\red}[1]{\mathbf{#1}}

\begin{document}

\title{\textsc{Maximum mean discrepancies of Farey sequences}}

\author{Toni Karvonen\textsuperscript{1} --- Anatoly Zhigljavsky\textsuperscript{2} \vspace{0.2cm} \\ \textsuperscript{1}\emph{Lappeenranta--Lahti University of Technology LUT, Finland.} \\ \textsuperscript{2}\emph{Cardiff University, United Kingdom}}

\maketitle

\begin{abstract}
  \noindent We identify a large class of positive-semidefinite kernels for which a certain polynomial rate of convergence of maximum mean discrepancies of Farey sequences is equivalent to the Riemann hypothesis. This class includes all Matérn kernels of order at least one-half.
\end{abstract}

\section{Introduction}

A kernel $K \colon \Omega \times \Omega \to \mathbb{R}$ on a domain $\Omega$ is positive-semidefinite if the kernel Gram matrix with elements $K(x_i, x_j)$ for $i,j = 1, \ldots, N$ is positive-semidefinite for all $N \in \mathbb{N}$ and $X = (x_1, \ldots, x_N) \subseteq \Omega$.
Each positive-semidefinite kernel induces a unique reproducing kernel Hilbert space $H$ (RKHS) that is a Hilbert space of certain real-valued functions defined on $\Omega$ equipped with a norm $\lVert \cdot \rVert_H$~\citep{Berlinet2004, Paulsen2016}.
Properties such as differentiability and boundedness of the kernel determine which functions are elements of $H$.
Under mild measurability assumptions one can then define the \emph{maximum mean discrepancy} (MMD)
\begin{equation*}
  \mathrm{MMD}(P, X) = \sup_{ \lVert f \rVert_H \leq 1 } \bigg\lvert \int_\Omega f(x) P(\dif x) - \frac{1}{N} \sum_{i=1}^N f(x_i) \bigg \rvert 
\end{equation*}
between a probability measure $P$ on $\Omega$ and the empirical measure $\xi=\sfrac{1}{N} \sum_{i=1}^N \delta_{x_i}$  of the point set $X$.
The MMD measures how well the empirical measure approximates $P$ or, in other words, how well the points are $P$-distributed.
Among other things, the MMD is routinely used as a test statistic and for hypothesis testing in non-parametric statistics and machine learning~\citep{Gretton2012, Sejdinovic2013}.

The Riemann hypothesis (RH) states that the non-trivial roots of the Riemann zeta function $\zeta(s) = \sum_{k=1}^\infty k^{-s}$ on the complex plane all have real part $\sfrac{1}{2}$.
The purpose of this note is to show that the RH is equivalent to a certain polynomial rate of convergence of the MMD of the uniform measure on $[0, 1]$ and the empirical measure of the \emph{Farey sequence}.
The equivalence holds for a large class of commonly used kernels, including Matérns.
The $n$th Farey sequence 
$F_n = (a_{i,n})_{i=1}^N$ 
is the increasing sequence of  reduced fractions on $[0, 1]$ whose denominators do not exceed $n$.
The first six Farey sequences are 
\begin{align*}
  F_1 &= \bigg( \red{\frac{0}{1}}, \: \red{\frac{1}{1}} \bigg)  = (a_{1,1}, \: a_{2,1}), \\
  F_2 &= \bigg( \frac{0}{1}, \: \red{\frac{1}{2}}, \: \frac{1}{1} \bigg) = (a_{1,2}, \: a_{2,2}, \: a_{3,2}), \\
  F_3 &= \bigg( \frac{0}{1}, \: \red{\frac{1}{3}}, \: \frac{1}{2}, \: \red{\frac{2}{3}}, \: \frac{1}{1} \bigg) = (a_{1,3}, \: a_{2,3}, \: a_{3,3}, \: a_{4,3}, \: a_{5,3}), \\
  F_4 &= \bigg( \frac{0}{1}, \: \red{\frac{1}{4}}, \: \frac{1}{3}, \: \frac{1}{2}, \: \frac{2}{3}, \: \red{\frac{3}{4}}, \: \frac{1}{1} \bigg) = (a_{1,4}, \: a_{2,4}, \: a_{3,4}, \: a_{4,4}, \: a_{5,4}, \: a_{6,4}, \: a_{7,4}) , \\
  F_5 &= \bigg( \frac{0}{1}, \: \red{\frac{1}{5}}, \: \frac{1}{4}, \: \frac{1}{3}, \: \red{\frac{2}{5}}, \: \frac{1}{2}, \: \red{\frac{3}{5}}, \: \frac{2}{3}, \: \frac{3}{4}, \: \red{\frac{4}{5}}, \: \frac{1}{1} \bigg) = (a_{1,5}, \: a_{2,5}, \ldots , \: a_{10, 5}, \: a_{11,5}), \\
  F_6 &= \bigg( \frac{0}{1}, \: \red{\frac{1}{6}}, \: \frac{1}{5}, \: \frac{1}{4}, \: \frac{1}{3}, \: \frac{2}{5}, \: \frac{1}{2}, \: \frac{3}{5}, \: \frac{2}{3}, \: \frac{3}{4}, \: \frac{4}{5}, \: \red{\frac{5}{6}}, \: \frac{1}{1} \bigg) = (a_{1,6}, \: a_{2,6}, \ldots , \: a_{12,6}, \: a_{13,6}),
\end{align*}
where we have bolded the points in $F_n$ that do not appear in $F_{n-1}$.
The first 18 Farey sequences are shown in Figure~\ref{Fig:Farey}.
Note that $n$ refers to the index of a Farey sequence, not to the number of points in the $n$th Farey sequence, which is $N = \lvert F_n \rvert = \Phi(n)+1$,  where $\Phi$ is the summatory  totient function.

Farey sequences have long history and many appearances in different branches of mathematics~\cite{cobeli2003haros}. The most intriguing appearance of Farey sequences is related to their connections with the RH, which originate with a classical result by Franel~\cite{Franel1924}, who proved that 
\begin{equation} \label{Eq:Franel}
  \sum_{i=1}^N \bigg( \frac{i}{N} - a_{i,n} \bigg)^2 = O(n^{-1+\varepsilon}) \qquad \text{ as } \qquad n \to \infty
\end{equation}
for every $\varepsilon > 0$ is equivalent to the RH.
Equation~\eqref{Eq:Franel} is a statement about the uniformity of the distribution of Farey sequences.
We refer to~\cite{Mikolas1992} for a comprehensive review.
In this note we exploit these connections and recast Franel's result using the statistical concept of MMD, hence providing an equivalent formulation of the RH with statistical flavour.

\begin{figure}
  
  \centering
  \includegraphics[width=0.48\textwidth]{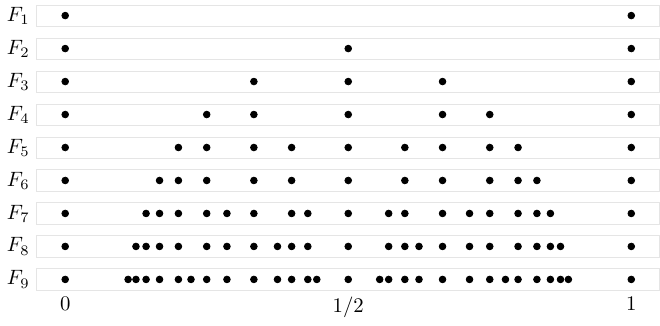}
  \hfill
  \includegraphics[width=0.48\textwidth]{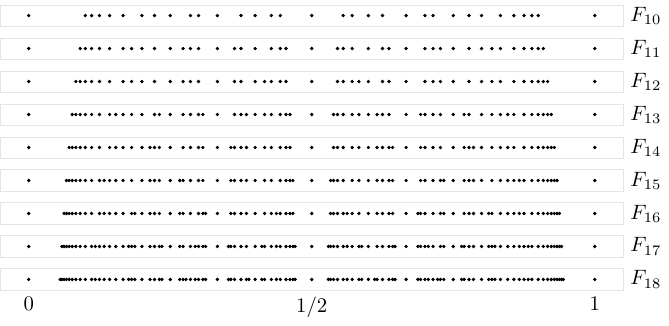}
  \caption{The first 18 Farey sequences $F_1, \ldots, F_{18}$.}
  \label{Fig:Farey}
  
\end{figure}

Equation~\eqref{Eq:Franel} can be understood as a statement about the discretised $L^2$-discrepancy; see~\cite{dress1999discrepance, Niederreiter1973} for results on the star-discrepancy.
Let $X = (x_1, \ldots, x_N) \subset [0, 1]$ be an increasing sequence.
The local discrepancy of $X$ on $[0, 1]$ is given by
\begin{equation*}
  D_\textup{loc}(A, X) = \bigg\lvert \frac{\lvert A \cap X \rvert}{N} - \mathrm{meas}(A) \bigg\lvert \qquad \text{ for measurable } \qquad A \subseteq [0, 1].
\end{equation*}
The $L^2$-discrepancy measures the uniformity of $X$ and is defined as
\begin{equation*}
  D^2(X) = \sqrt{ \int_0^1 D_\textup{loc}([0, \alpha], X)^2 \dif \alpha }.
\end{equation*}
An approximation $\tilde{D}^2(X)$ to $D^2(X)$ can be obtained by discretising at $x_1, \ldots, x_N$:
\begin{equation*}
  \tilde{D}^2(X) = \sqrt{ \frac{1}{N} \sum_{i=1}^N D_\textup{loc}([0, x_i], X)^2 } = \sqrt{ \frac{1}{N} \sum_{i=1}^N \bigg( \frac{i}{N} - x_i \bigg)^2 } \approx D^2(X) ,
\end{equation*}
where the second equality follows from the facts that $\mathrm{meas}([0,x_i]) = x_i$ and, since $X$ is an increasing sequence, $\lvert [0, x_i] \cap X \rvert = i$.
A theorem of Mertens~\cite[§ 18.5]{HardyWright1960} gives the asymptotic equivalence
\begin{equation} \label{Eq:Mertens}
  N = \lvert F_n \rvert \sim \frac{n^2}{2 \zeta(2)} = \frac{3n^2}{\pi^2} ,
\end{equation}
where $\zeta(2) = \pi^2 / 6$ is the famous value of the Riemann zeta function at $s = 2$ originally computed by Euler.
Franel's result in Equation~\eqref{Eq:Franel} and the asymptotics in Equation~\eqref{Eq:Mertens} then show that the RH is equivalent to
\begin{equation*}
  \tilde{D}^2(F_n) = \sqrt{ \frac{1}{N} \sum_{i=1}^N \bigg( \frac{i}{N} - a_{i,n} \bigg)^2 } = O(n^{-3/2+\varepsilon}) = O(N^{-3/4+\varepsilon})
\end{equation*}
for every $\varepsilon > 0$.
That is, the RH is equivalent to a statement about the discretised $L^2$-discrepancies of Farey sequences.
It should now come as no suprise that the RH can also be formulated in terms of the MMD.

\section{Results and remarks} \label{Sec:main}

For the MMD between the uniform measure on $[0, 1]$ and the empirical measure for a sequence points $X = (x_1, \ldots, x_N) \subset [0, 1]$ we use the simplified notation
\begin{equation} \label{Eq:MMD-01}
        \mathrm{MMD}(X) = \sup_{ \lVert f \rVert_H \leq 1 } \bigg\lvert \int_0^1 f(x) \dif x - \frac{1}{N} \sum_{i=1}^N f(x_i) \bigg \rvert .
\end{equation}
The following theorem connects the rate of convergence of MMDs of Farey sequences to the RH.
We give the proof in Section~\ref{Sec:Proofs}.
 
\begin{theorem} \label{Thm:Main}
  Let $K$ be a positive-semidefinite kernel on $[0, 1]$ and $H$ its RKHS.
  Suppose that
  \begin{enumerate}
  \item[(a)] $H$ is a subset of $W^{1,2}([0, 1])$, the Sobolev space of order one on $[0, 1]$, and
  \item[(b)] $H$ contains the function $x \mapsto a + b x + x^\beta$ for some $a, b \in \mathbb{R}$ and $\beta \in [2, \gamma] \cup \{4, 5\}$, where $\gamma = 1 + \sfrac{6}{\sqrt{3}D} \approx 3.405$ and $D = \sfrac{\pi^2}{6} + \sfrac{2}{3} \log 2 - \sfrac{2}{3}$. 
  \end{enumerate}
  Then the RH is equivalent to
  \begin{equation} \label{Eq:RH}
    \mathrm{MMD}(F_n) = O(n^{-3/2 + \varepsilon}) = O(N^{-3/4 + \varepsilon}) \qquad \text{ for every } \qquad \varepsilon > 0.
  \end{equation}
\end{theorem}

Note that assumption~(b) is satisfied if $H$ contains the monomial $x \mapsto x^\beta$ for some $\beta \in \{2, 3, 4, 5\}$.
Figure~\ref{Fig:Farey-MMD} shows how the MMD behaves.
There are three commonly used classes of kernels that are covered by Theorem~\ref{Thm:Main}:
\begin{enumerate}
  \item Let $\lambda > 0$ be a correlation length parameter.
The Matérn kernel of order $\nu > 0$ is given by
\begin{equation} \label{Eq:Matern}
  K_\nu(x, y) = \frac{2^{1-\nu}}{\Gamma(\nu)} \bigg( \frac{\sqrt{2\nu} \lvert x - y \rvert }{\lambda} \bigg)^\nu \mathcal{K}_\nu \bigg( \frac{\sqrt{2\nu} \lvert x - y \rvert }{\lambda} \bigg) ,
\end{equation}
where $\Gamma$ is the gamma function and $\mathcal{K}_\nu$ the modified Bessel function of the second kind of order $\nu$.
Matérn kernels are widely used to define Gaussian random field models in spatial statistics~\cite{Stein1999}.
It is well known that the RKHS of a Matérn kernel of order $\nu$ on $[0, 1]$ is norm-equivalent to the (possibly fractional) Sobolev space $W^{\nu + 1/2,2}([0, 1])$; see~\cite[Thm.\@~6.13 and Cor.\@~10.48]{Wendland2005}.
Assumptions~(a) and~(b) are thus satisfied if $\nu \geq \sfrac{1}{2}$ because Sobolev spaces are nested and contain all polynomials.
\item By integrating $\min\{x, y\}$, the covariance kernel of Brownian motion, $m$ times and adding a polynomial part that removes boundary conditions at the origin one obtains the released $m$-fold integrated Brownian motion kernel
\begin{equation} \label{Eq:Int-BM}
  K_m(x, y) = \sum_{k=0}^m \frac{(xy)^k}{(k!)^2} + \int_0^1 \frac{(x - t)_+^m \, (y - t)_+^m}{(m!)^2} \dif t ,
\end{equation}
where $(x)_+ = \max\{0, x\}$.
Because its RKHS is norm-equivalent to $W^{m+1,2}([0,1])$, this kernel satisfies the assumptions of Theorem~\ref{Thm:Main}~\cite[p.\@~322]{Berlinet2004}.
\item The energy-distance kernel 
\begin{equation} \label{eq:energy-distance-K}
 K_\alpha(x, y)= |x|^\alpha+|y|^\alpha-
|x-y|^\alpha
\end{equation}
is positive-semidefinite for $\alpha \in (0, 2)$.
The energy-distance kernel is, up to scaling, the covariance kernel of the fractional Brownian motion with Hurst index $\sfrac{\beta}{2}$; for $\alpha = 1$ it reduces to the Brownian motion kernel.
In Section~\ref{Sec:energy-distance-proof} we use a characterisation by Barton and Poor~\cite{BartonPoor1988} to verify that the RKHS of an energy-distance kernel satisfies the assumptions of Theorem~\ref{Thm:Main} for every $\alpha \in [1, 2)$.
\end{enumerate}

\begin{figure}[t]
  
  \centering
  \includegraphics[width=\textwidth]{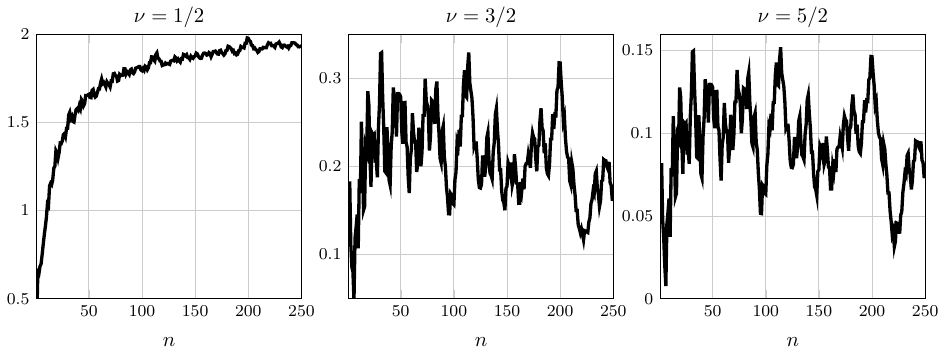}
  \caption{The plots show $\mathrm{MMD}(F_n) \cdot n^{3/2}$, the normalised MMDs of Farey sequences, up to $n = 250$ for Matérn kernels with (i) $\nu = 1/2$ and $\lambda = 1$, (ii) $\nu = 3/2$ and $\lambda = 3^{1/2}$, and (iii) $\nu = 5/2$ and $\lambda = 5^{1/2}$. Note that $n$ is the index of the Farey sequence, not the number of points. For $n = 250$ we have $N = \lvert F_n \rvert = 19,\!025$.}
  \label{Fig:Farey-MMD}
  
\end{figure}

We summarise these observations in the following proposition.

\begin{proposition} \label{Cor:Main}
Suppose that $K$ is (i) any Matérn kernel of order $\nu \geq \sfrac{1}{2}$ in Equation~\eqref{Eq:Matern}, (ii) a released $m$-fold integrated Brownian motion kernel in Equation~\eqref{Eq:Int-BM} with $m \geq 0$, or (iii) an energy-distance kernel in Equation~\eqref{eq:energy-distance-K} with $\alpha \in [1, 2)$. Then the RH is equivalent to Equation~\eqref{Eq:RH}.
\end{proposition}

\begin{remark}
Matérn kernels, integrated Brownian motion kernels, and energy-distance kernel are only finitely differentiable.
Theorem~\ref{Thm:Main} applies also to some infinitely differentiable kernels, such as $K(x, y) = \exp(xy)$.
The RKHS of this particular kernel is contained in every Sobolev space and includes all polynomials by the virtue of the Taylor expansion $K(x, y) = \sum_{k=0}^\infty (xy)^k / k!$; see \cite[Sec.\@~2.1]{Paulsen2016} and~\cite{ZwicknaglSchaback2013}.
The theorem does not cover the popular Gaussian kernel $K(x, y) = \exp(-(x-y)^2/(2\lambda^2))$ because its RKHS does not contain any non-trivial polynomials~\cite{Minh2010}.
\end{remark}

\begin{remark}
There are a number of results which state that the RH is equivalent to
\begin{equation*}
  \bigg\lvert \int_0^1 f(x) \dif x - \frac{1}{N} \sum_{i=1}^N f(a_{i,n}) \bigg\rvert = O(n^{-3/2 + \varepsilon})
\end{equation*}
for every $\varepsilon > 0$ if $f$ is a suitable fixed function.
The proof of Theorem~\ref{Thm:Main} uses such results from~\cite{Mikolas1949, yoshimoto1998farey} that apply to certain monomials.
Other options from~\cite{Mikolas1949, Mikolas1951, yoshimoto1998farey} include functions whose derivatives satisfy a particular inequality and certain functions with singularity at the origin, such as $f(x) = \log(x)$.
Kanemitsu and Yoshimoto~\cite{Kanemitsu2000} provide a collection of suitable functions expressed as Fourier cosine series.
\end{remark}

\section{Implications for energy-distance kernels}

Theorem~\ref{Thm:Main} has some curious implications when applied to the energy-distance kernel $K_\alpha$ in Equation~\eqref{eq:energy-distance-K}.
By direct calculation, we obtain the following expression for the squared MMD of the $n$th Farey sequence:
\begin{equation}
\label{eq:mmd2}
\mathrm{MMD}^2(F_n)=\frac{4}{(\alpha+1)N}
\sum_{i=1}^N  
a_{i,n}^{\alpha+1} -\frac{1}{N^2}\sum_{i,j=1}^N |a_{i,n}-a_{j,n}|^\alpha  - \frac{2}{(\alpha+1)(\alpha+2)}
\, .
\end{equation}
For $\alpha = 1$, the energy-distance kernel reduces to the Brownian motion kernel $K_1(x, y) = \min\{x, y\}$ that we use in the proof of Theorem~\ref{Thm:Main} (in particular, see Lemma~\ref{Lemma:BM-MMD}).
Here we consider $\alpha \in (1, 2)$.
Denote
\begin{equation*}
S_{\alpha,N}=\frac{1}{N}
\sum_{i=1}^N  
a_{i,n}^{\alpha+1}  \quad \text{ and} \quad T_{\alpha,N}= \frac{1}{N^2}\sum_{i,j=1}^N |a_{i,n}-a_{j,n}|^\alpha ,
\end{equation*}
so that
\begin{equation*}
  \mathrm{MMD}^2(F_n) = \bigg( \frac{4}{\alpha + 1} S_{\alpha,N} - \frac{4}{(\alpha+1)(\alpha+2)} \bigg) + \bigg( \frac{2}{(\alpha+1)(\alpha+2)} - T_{\alpha,N} \bigg) \eqqcolon \delta_{1,N} + \delta_{2,N} .
\end{equation*}
By the asymptotic uniformity of $F_n$ [see Equation~\eqref{Eq:Mikolas}],
\begin{equation*}
  S_{\alpha,N} \to \int_0^1 x^{\alpha+1} \dif x = \frac{1}{\alpha + 2} 
\end{equation*}
with the rate $O(n^{-3/2 + \varepsilon})$ for every $\varepsilon > 0$ as $N \to \infty$.
That is, $\delta_{1,N} = O(n^{-3/2 + \varepsilon})$.
However, Theorem~\ref{Thm:Main} and Proposition~\ref{Cor:Main} state that the RH is equivalent to $\mathrm{MMD}^2(F_n) = O(n^{-3 + \varepsilon})$ for any $\alpha \in [1, 2)$.
This has two consequences.
First, 
\begin{equation*}
  \begin{split}
  \delta_{2,N} = \frac{2}{(\alpha+1)(\alpha+2)} - T_{\alpha,N} &= \int_0^1 \int_0^1 \lvert x - y \rvert^\alpha \dif x \dif y - \frac{1}{N^2}\sum_{i,j=1}^N |a_{i,n}-a_{j,n}|^\alpha \\
  &= O(n^{-3/2+\varepsilon}) 
  \end{split}
\end{equation*}
if the RH holds.
Second, the RH requires there to be substantial cancellations between the terms $\delta_{1,N}$ and $\delta_{2,N}$ which individually tend to zero as $O(n^{-3/2 + \varepsilon})$ but whose sum must have the much faster rate $O(n^{-3+\varepsilon})$ under the RH.


\section{Proofs} \label{Sec:Proofs}


This section contains proofs for the results in Section~\ref{Sec:main}.

\subsection{Proof of Theorem~\ref{Thm:Main}}

The proof of Theorem~\ref{Thm:Main} uses the following lemma, which in fact states that the MMD for a piecewise linear kernel equals the $L^2$-discrepancy~\citep{Dick2014, Warnock72}.

\begin{lemma} \label{Lemma:BM-MMD}
  Let $N \in \mathbb{N}$ be odd and $K(x, y) = 1 + \min\{x, y\}$. If an increasing sequence $X = (x_1, \ldots, x_N)$ contains $\sfrac{1}{2}$ and is symmetric on $[0, 1]$, in that $1 - x \in X$ if $x \in X$, then
  \begin{equation*}
    \mathrm{MMD}(X)^2 = \frac{1}{N} \sum_{i=1}^N \bigg( \frac{i}{N} - x_i \bigg)^2 - \frac{1}{6N^2} .
  \end{equation*}
\end{lemma}
\begin{proof}
From the reproducing property of the kernel in $H$ it follows that the squared MMD admits the closed form expression
\begin{equation} \label{Eq:MMD-Explicit}
  \mathrm{MMD}(X)^2 = \int_0^1 \int_0^1 K(x, y) \dif x \dif y - \frac{2}{N} \sum_{i=1}^N \int_0^1 K(x, x_i) \dif x + \frac{1}{N^2} \sum_{i,j=1}^N K(x_i, x_j) 
\end{equation}
in terms of the kernel and its integrals; see, for example, \cite[Lem.\@~6]{Gretton2012} or \cite[Sec.\@~10.2]{NovakWozniakowski2010}.
For the kernel $K(x, y) = 1 + \min\{x, y\}$ it is straightforward to compute that
\begin{equation*}
  \int_0^1 K(x, y) \dif x = 1 + \frac{1}{2} (2 - y) y \quad \text{ and } \quad \int_0^1 \int_0^1 K(x, y) \dif x \dif y = \frac{4}{3}.
\end{equation*}
Therefore Equation~\eqref{Eq:MMD-Explicit} gives
\begin{align*}
  \mathrm{MMD}(X)^2 &= \frac{4}{3} - \frac{2}{N} \sum_{i=1}^N \bigg( 1 + \frac{1}{2}(2-x_i)x_i \bigg) + \frac{1}{N^2} \sum_{i,j=1}^N (1 + \min\{x_i, x_j\}) \\
        &= \frac{1}{3} - \frac{1}{N} \sum_{i=1}^N (2 - x_i) x_i + \frac{1}{N^2} \sum_{i=1}^N (2N - 2i + 1) x_i \\
        &= \frac{1}{N} \sum_{i=1}^N \bigg(\frac{i}{N} - x_i \bigg)^2 + \frac{1}{N^2} \sum_{i=1}^N x_i - \frac{3N + 1}{6N^2}.
\end{align*}
Because the sequence $X$ contains $\sfrac{1}{2}$ and is symmetric,
\begin{equation*}
  \frac{1}{N^2} \sum_{i=1}^N x_i - \frac{3N + 1}{6N^2} = \frac{1}{N^2} \bigg( \frac{1}{2} + \frac{N-1}{2} \bigg) - \frac{3N + 1}{6N^2} = -\frac{1}{6N^2} .
\end{equation*}
This concludes the proof.
\end{proof}

\begin{proof}[Proof of Theorem~\ref{Thm:Main}]
Consider first the kernel $K_0(x, y) = 1 + \min\{x, y\}$ from Lemma~\ref{Lemma:BM-MMD}.
Because they contain $\sfrac{1}{2}$ for $n \geq 2$ and are symmetric, we may apply the lemma to the Farey sequences $F_n$ and so obtain
\begin{equation} \label{Eq:MMD-BM-Fn}
  \mathrm{MMD}_0(F_n)^2 = \frac{1}{N} \sum_{i=1}^N \bigg( \frac{i}{N} - a_{i,n} \bigg)^2 - \frac{1}{6N^2} 
\end{equation}
for the $K_0$-MMD of $F_n$ when $n \geq 2$.
The result by Franel in Equation~\eqref{Eq:Franel} and the asymptotics $N = \lvert F_n \rvert \sim 3n^2/\pi^2$ from Equation~\eqref{Eq:Mertens} imply that the RH is equivalent to
\begin{equation*}
  \sum_{i=1}^N \bigg( \frac{i}{N} - a_{i,n} \bigg)^2 = O(n^{-1 + \varepsilon}) = O(N^{-1/2 + \varepsilon})
\end{equation*}
for every $\varepsilon > 0$.
Inserting this in Equation~\eqref{Eq:MMD-BM-Fn} and observing that the second term is negligible shows that the RH is equivalent to $\mathrm{MMD}_0(F_n) = O(N^{-3/4 + \varepsilon})$, from which Equation~\eqref{Eq:RH} follows for $K = K_0$ by using Equation~\eqref{Eq:Mertens} again.

It is well known~(e.g., \cite[Sec.\@~3.1]{AdlerTaylor2007} or \cite[Sec.\@~10]{VaartZanten2008}) that the RKHS $H_0$ of $K_0(x, y) = 1 + \min\{x, y\}$ is norm-equivalent to the Sobolev space $W^{1,2}([0,1])$.
By an RKHS inclusion theorem of Aronszajn~\cite[Thm.\@~5.1]{Paulsen2016}, the norm of any RKHS $H \subseteq H_0$ satisfies $\lVert f \rVert_{H_0} \leq c \lVert f \rVert_{H}$ for a positive $c$ and all $f \in H$.
From this it follows that the unit ball of $H$ is contained in the $c$-ball of $H_0$.
Because the MMD in Equation~\eqref{Eq:MMD-01} is a supremum over the unit ball of the RKHS, the MMD of $F_n$ for any kernel $K$ that satisfies assumption~(a) is bounded from above by a constant multiple of the MMD of $F_n$ for $K_0$.
Consequently, 
\begin{equation*}
  \mathrm{MMD}(F_n) = O(\,\mathrm{MMD}_0(F_n)\,) = O(n^{-3/2 + \varepsilon}) = O(N^{-3/4 + \varepsilon}) 
\end{equation*}
for every $\varepsilon > 0$ if the RH is true.
To show that this rate implies the RH we use assumption~(b).
Let $f_\beta(x) = a + bx + x^\beta$ for $a,b \in \mathbb{R}$ and $\beta \in [2, \gamma] \cup \{4, 5\}$, where 
\begin{equation*}
  \gamma = 1 + \frac{6}{\sqrt{3} D} \approx 3.405 \quad \text{ and } \quad D = \frac{\pi^2}{6} + \frac{2}{3} \log 2 - \frac{2}{3} .
\end{equation*}
Because Farey sequences are symmetric about $\sfrac{1}{2}$,
\begin{equation*}
  \bigg\lvert \int_0^1 f_\beta(x) \dif x - \frac{1}{N} \sum_{i=1}^N f(a_{i,n}) \bigg\rvert = \bigg\lvert \int_0^1 x^\beta \dif x - \frac{1}{N} \sum_{i=1}^N a_{i,n}^\beta \bigg\rvert.
\end{equation*}
Results by Mikolás~\cite[Thm.\@~5]{Mikolas1949} and Yoshimoto~\cite[pp.\@~302--3]{yoshimoto1998farey} state, in combination with the asymptotics in Equation~\eqref{Eq:Mertens}, that the RH is equivalent to
\begin{equation} \label{Eq:Mikolas}
  \bigg\lvert \int_0^1 x^\beta \dif x - \frac{1}{N} \sum_{i=1}^N a_{i,n}^\beta \bigg\rvert = O(n^{-3/2 + \varepsilon}) = O(N^{-3/4 + \varepsilon}) 
\end{equation}
for every $\varepsilon > 0$.
Under assumption~(b) the unit ball of $H$ contains a constant multiple of $f_\beta$, so that the rate $\mathrm{MMD}(F_n) = O(n^{-3/2 + \varepsilon})$ implies Equation~\eqref{Eq:Mikolas} and consequently also the RH.
This concludes the proof.
\end{proof}

\subsection{Proof of Proposition~\ref{Cor:Main} for energy-distance kernels} \label{Sec:energy-distance-proof}

We first prove that RKHSs of energy-distance kernels contain suitable polynomials.

\begin{lemma} \label{lemma:energy-distance-rkhs}
  Let $m$ be a positive integer.
  The RKHS of the energy-distance kernel $K_\alpha$ in~\eqref{eq:energy-distance-K} on $[0, 1]$ contains a polynomial of degree $m$ for every $\alpha \in (0, 2)$.
\end{lemma}
\begin{proof}
  Let $H_\alpha(\mathbb{R})$ denote the RKHS of the energy-distance kernel~\eqref{eq:energy-distance-K} on $\mathbb{R}$.
  Theorem~4.1 in~\cite{BartonPoor1988} states that $f \in H_\alpha(\mathbb{R})$ if and only if there is $g \colon \mathbb{R} \to \mathbb{C}$ satisfying $\int_\mathbb{R} \lvert g(\omega) \rvert^2 \lvert \omega \rvert^{1 - \alpha} \dif \omega < \infty$ such that 
  \begin{equation} \label{eq:f-barton-poor}
    f(x) = \int_\mathbb{R} \frac{e^{i \omega x} - 1}{i \omega} \, \overline{g(\omega)} \lvert \omega \rvert^{1 - \alpha} \dif \omega \quad \text{ for all } \quad x \in \mathbb{R}.
  \end{equation}
  We select
  \begin{equation*}
    g(\omega) = -i \omega \, \mathrm{sinc}^{m+1} \bigg( \frac{\omega}{2\pi} \bigg) \frac{1}{\lvert \omega \rvert^{1-\alpha}} .
  \end{equation*}
  For this function the integral
    \begin{equation*}
    \int_\mathbb{R} \lvert g(\omega) \rvert^2 \lvert \omega \rvert^{1 - \alpha} \dif \omega = \int_\mathbb{R} \lvert \omega \rvert^{1 + \alpha} \mathrm{sinc}^{2(m+1)} \bigg( \frac{\omega}{2\pi} \bigg) \dif \omega = \int_\mathbb{R} \frac{\sin^{2(m+1)}(\omega/2\pi)}{\lvert \omega \rvert^{2m+1 - \alpha}} \dif \omega
  \end{equation*}
  is finite if $2m > \alpha$.
  Therefore the function $f$ defined in Equation~\eqref{eq:f-barton-poor} is in $H_\alpha(\mathbb{R})$ for any $\alpha \in (0, 2)$.
  This function is
  \begin{equation*}
    \begin{split}
    f(x) &= \int_\mathbb{R} \bigg[ e^{i\omega x} \mathrm{sinc}^{m+1}  \bigg( \frac{\omega}{2\pi} \bigg) - \mathrm{sinc}^{m+1}  \bigg( \frac{\omega}{2\pi} \bigg) \bigg]\dif \omega \\
    &= \int_{\mathbb{R}} [\cos(x \omega) + i \sin(x\omega) ] \, \mathrm{sinc}^{m+1} \bigg( \frac{\omega}{2\pi} \bigg) \dif \omega - \mathrm{constant} \\
        &= \int_{\mathbb{R}} \cos (x \omega) \, \mathrm{sinc}^{m+1} \bigg( \frac{\omega}{2\pi} \bigg) \dif \omega - \mathrm{constant},
    \end{split}
  \end{equation*}
  where the integral is the cosine transform of the $(m+1)$th power of $\operatorname{sinc}$.
  By Equation~(16) in Section~1.6 of~\cite{Erdelyi1954}, this cosine transform equals a polynomial of degree $m$ on some non-empty interval $[0, \delta]$.
  By applying a suitable scaling we obtain a function in $H_\alpha(\mathbb{R})$ that is a polynomial of degree $m$ on $[0, 1]$ with a leading coefficient one.
  This proves the claim because the RKHS of $K_\alpha$ on $[0, 1]$ consists of restrictions onto $[0, 1]$ of functions in $H_\alpha(\mathbb{R})$~\cite[p.\@~25]{Berlinet2004}.
\end{proof}

\begin{proof}[Proof of Proposition~\ref{Cor:Main} for energy-distance kernels with $\alpha \in [1, 2)$]
  Assumption~(b) holds by setting $m = 2$ in Lemma~\ref{lemma:energy-distance-rkhs}.
  For $\alpha = 1$ and $x, y \geq 0$ the energy-distance kernel becomes
  \begin{equation*}
    K_1(x, y) = \lvert x \rvert + \lvert y \rvert - \lvert x - y \rvert = \min\{x, y\} .
  \end{equation*}
  In the proof of Theorem~\ref{Thm:Main} we noted that the RKHS of $K(x, y) = 1 + \min\{x, y\}$ on $[0, 1]$ is norm-equivalent to $W^{1,2}([0,1])$.
  The RKHS of $K$ consists of sums of constant functions with elements of the RKHS of $K_1$~\cite[p.\@~24]{Berlinet2004}.
  From the characterisation of $H_\alpha(\mathbb{R})$ in~\cite{BartonPoor1988} that we used in the proof of Lemma~\ref{lemma:energy-distance-rkhs} it also follows that $H_\alpha(\mathbb{R}) \subseteq H_\gamma(\mathbb{R})$ if $\alpha \geq \gamma$.
  This inclusion is inherited by RKHSs on $[0, 1]$.
  Therefore $K_\alpha$ satisfies assumption~(a) for $\alpha \in [1, 2)$.
\end{proof}

\section*{Acknowledgements}

TK was supported by the Research Council of Finland projects 338567 (``Scalable, adaptive and reliable probabilistic integration''), 359183 (``Flagship of Advanced Mathematics for Sensing, Imaging and Modelling''), and 368086 (``Inference and approximation under misspecification'').

\setlength{\bibsep}{2   pt plus 0.3ex} 

\bibliographystyle{abbrv}

\begin{thebibliography}{10}

\bibitem{AdlerTaylor2007}
R.~J. Adler and J.~E. Taylor.
\newblock {\em Random Fields and Geometry}.
\newblock Springer Monographs in Mathematics. Springer, 2007.

\bibitem{BartonPoor1988}
R.~J. Barton and H.~V. Poor.
\newblock Signal detection in fractional {G}aussian noise.
\newblock {\em IEEE Trans. Autom. Control}, 34(5):943--959, 1988.

\bibitem{Berlinet2004}
A.~Berlinet and C.~Thomas{-}Agnan.
\newblock {\em Reproducing Kernel {H}ilbert Spaces in Probability and Statistics}.
\newblock Springer, 2004.

\bibitem{cobeli2003haros}
C.~Cobeli and A.~Zaharescu.
\newblock The {H}aros-{F}arey sequence at two hundred years.
\newblock {\em Acta Univ.\ Apulensis. Math. - Inf.}, 5:1--38, 2003.

\bibitem{Dick2014}
J.~Dick and F.~Pillichshammer.
\newblock Discrepancy theory and quasi-{M}onte {C}arlo integration.
\newblock In {\em A Panorama of Discrepancy Theory}, pages 539--619. Springer Cham, 2014.

\bibitem{dress1999discrepance}
F.~Dress.
\newblock Discr{\'e}pance des suites de {F}arey.
\newblock {\em J. Théor. Nr. Bordx}, 11(2):345--367, 1999.

\bibitem{Erdelyi1954}
A.~Erd{\'e}lyi, W.~Magnus, F.~Oberhettinger, and F.~G. Tricomi.
\newblock {\em Tables of Integral Transforms}, volume~1.
\newblock McGraw-Hill, 1954.

\bibitem{Franel1924}
J.~Franel.
\newblock Les suites de {F}arey et le problème des nombres premiers.
\newblock {\em Nachr. Ges. Wiss. Göttingen}, pages 198--201, 1924.

\bibitem{Gretton2012}
A.~Gretton, K.~M. Borgwardt, M.~J. Rasch, B.~Schölkopf, and A.~Smola.
\newblock A kernel two-sample test.
\newblock {\em J. Mach. Learn. Res.}, 13:723--773, 2012.

\bibitem{HardyWright1960}
G.~H. Hardy and E.~M. Wright.
\newblock {\em An Introduction to the Theory of Numbers}.
\newblock Oxford University Press, 4th edition, 1960.

\bibitem{Kanemitsu2000}
S.~Kanemitsu and M.~Yoshimoto.
\newblock {E}uler products, {F}arey series and the {R}iemann hypothesis.
\newblock {\em Publ. Math. Debrecen}, 56(3--4):431--449, 2000.

\bibitem{Mikolas1949}
M.~Mikolás.
\newblock {F}arey series and their connection with the prime number problem. {I}.
\newblock {\em Acta Sci. Math. (Szeged)}, 13:93--117, 1949.

\bibitem{Mikolas1951}
M.~Mikolás.
\newblock {F}arey series and their connection with the prime number problem. {II}.
\newblock {\em Acta Sci. Math. (Szeged)}, 14:5--21, 1951.

\bibitem{Mikolas1992}
M.~Mikolás and K.~Sato.
\newblock On the asymptotic behaviour of {F}ranel's sum and the {R}iemann hypothesis.
\newblock {\em Results Math.}, 21:368--378, 1992.

\bibitem{Minh2010}
H.~Q. Minh.
\newblock Some properties of {G}aussian reproducing kernel {H}ilbert spaces and their implications for function approximation and learning theory.
\newblock {\em Constr. Approx.}, 32(2):307--338, 2010.

\bibitem{Niederreiter1973}
H.~Niederreiter.
\newblock The distribution of {F}arey points.
\newblock {\em Math. Ann.}, 201:341--345, 1973.

\bibitem{NovakWozniakowski2010}
E.~Novak and H.~Wo{\'z}niakowski.
\newblock {\em Tractability of Multivariate Problems. Volume {II}: Standard Information for Functionals}.
\newblock Number~12 in EMS Tracts in Mathematics. European Mathematical Society, 2010.

\bibitem{Paulsen2016}
V.~I. Paulsen and M.~Raghupathi.
\newblock {\em An Introduction to the Theory of Reproducing Kernel {H}ilbert Spaces}.
\newblock Number 152 in Cambridge Studies in Advanced Mathematics. Cambridge University Press, 2016.

\bibitem{Sejdinovic2013}
D.~Sejdinovic, B.~Sriperumbudur, A.~Gretton, and K.~Fukumizu.
\newblock Equivalence of distance-based and {RKHS}-based statistics in hypothesis testing.
\newblock {\em Ann. Stat.}, 41(5):2263--2291, 2013.

\bibitem{Stein1999}
M.~L. Stein.
\newblock {\em Interpolation of Spatial Data: Some Theory for Kriging}.
\newblock Springer Series in Statistics. Springer, 1999.

\bibitem{VaartZanten2008}
A.~W. van~der Vaart and J.~H. van Zanten.
\newblock Reproducing kernel {H}ilbert spaces of {G}aussian priors.
\newblock In {\em Pushing the Limits of Contemporary Statistics: Contributions in Honor of Jayanta K.\ Ghosh}, pages 200--222. Institute of Mathematical Statistics, 2008.

\bibitem{Warnock72}
T.~T. Warnock.
\newblock Computational investigations of low-discrepancy point sets.
\newblock In {\em Applications of Number Theory to Numerical Analysis}, pages 319--343. Elsevier, 1972.

\bibitem{Wendland2005}
H.~Wendland.
\newblock {\em Scattered Data Approximation}.
\newblock Number~17 in Cambridge Monographs on Applied and Computational Mathematics. Cambridge University Press, 2005.

\bibitem{yoshimoto1998farey}
M.~Yoshimoto.
\newblock Farey series and the {R}iemann hypothesis. {II}.
\newblock {\em Acta Math. Hungar.}, 78(4):287--304, 1998.

\bibitem{ZwicknaglSchaback2013}
B.~Zwicknagl and R.~Schaback.
\newblock Interpolation and approximation in {T}aylor spaces.
\newblock {\em J. Approx. Theory}, 171:65--83, 2013.

\end{thebibliography}

\end{document}